\documentclass[12pt]{article}
\usepackage{amsfonts,amsmath,amsxtra}

\newtheorem{theorem}{Theorem}
\newtheorem{corollary}{Corollary}
\newtheorem{lemma}{Lemma}

\newenvironment{remark}
{\smallskip\noindent{\bf Remark\/}.}{\smallskip\par}

\newenvironment{proof}
{\noindent{\bf Proof\/}.}{{ $\Box$}\smallskip\par}

\title{On the $S_n$-equivariant Euler characteristic of $M_{2,n}$.}
\author{E. Gorsky\footnote{Supported by the grant RFBR-007-00593.} }
\date{}

\begin{document}

\maketitle

\begin{abstract}
The Getzler's formula relates the $S_n$-equivariant Hodge-Deligne polynomial of the space of ordered tuples of distinct points on a given variety $X$ with the Hodge-Deligne polynomial of $X$. We obtain the analogue of this formula for the case when $X$ has a nontrivial automorphism group. Collecting together all strata of
$\mathcal{M}_2$ with different automorphism groups, we derive a formula for
the $S_n$-equivariant Euler characteristic of $\mathcal{M}_{2,n}$.
\end{abstract}

\section{Introduction}

For a quasiprojective algebraic variety we will denote by $B(X,n)$ the configuration space of unordered $n$-tuples of distinct points on $X$, by $S^{n}X$ the $n$-th symmetric power of $X$, by $F(X,n)$
the configuration space of ordered $n$-tuples of distinct points on $X$. It is clear that $B(X,n)$ is the quotient of $F(X,n)$ by the action of the symmetric group $S_n$.

The problem of calculation of the homology of $B(X,n)$ and $F(X,n)$ for a given $X$ is well known. For example, for $X=\mathbb{C}$ it was solved by Arnold (\cite{arnold}). Since one has an action of the symmetric group $S_n$ on $F(X,n)$, all (co)homology groups of $F(X,n)$ are representations of $S_n$. Therefore we can consider the $S_n$-equivariant Euler characteristic $\chi^{S_n}_{F(X,n)}$ (or $S_n$-equivariant Hodge-Deligne polynomial $e^{S_n}_{F(X,n)}$) of $F(X,n)$ --  an alternating sum (resp. generating function) of characters of these representations, belonging to the ring $\Lambda$ of symmetric polynomials of infinite number of variables.

E. Getzler (\cite{getzler}) proved, that
$$1+\sum_{n=1}^{\infty}t^{n}e^{S_n}_{F(X,n)}(u,v)=\prod_{n=1}^{\infty}(1+p_nt^{n})^{{1\over n}\sum_{m|n}
\mu(n/m)\Psi_m(e_X(u,v))},$$
where $\mu(k)$ is the Moebius function, $p_n$ is the $n$-th Newton polynomial of infinite number of variables, and $\Psi_m$ are Adams operations (see below).

For example, Getzler's formula implies that
$$1+\sum_{n=1}^{\infty}t^{n}\chi(B(X,n))=(1+t)^{\chi(X)}.$$

O. Tommasi has proved (\cite{tommasi}), that rational homologies of the moduli space of hyperelliptic curves of given arbitrary genus $g$ are trivial. Since every smooth genus 2 curve is hyperelliptic, homologies of the moduli space $\mathcal{M}_2$ of genus 2 curves are trivial too.
Using this fact and the Getzler's formula, we will derive  in Theorem 2 the $S_n$-equivariant Euler characteristic of the moduli space $\mathcal{M}_{2,n}$ of curves of genus 2 with $n$ marked points for all $n$.

The author is grateful to S. Gusein-Zade, M. Kazaryan and S. Lando
for the useful discussions and remarks.

\section{Adams operations}

Some rings, instead of the addition and the multiplication, carry an additional structure, namely a family of ring automorphisms
$\Psi_k$, labeled by natural numbers, such that
$$\Psi_k(x+y)=\Psi_k(x)+\Psi_k(y),\,\,\, \Psi_k(xy)=\Psi_k(x)\cdot\Psi_k(y),$$
$$\Psi_1(x)=x,\,\,\,\mbox{\rm and}\,\, \Psi_k(\Psi_m(x))=\Psi_{km}(x).$$
Such automorphisms are called the Adams operations.

One of the most natural examples of such operations can be constructed in the Grothendieck ring of representations of a finite group $G$. If $E$ is a (virtual) representation, one can consider its symmetric powers $S^{k}E$ and the series
$${d\over dt}\ln (1+S^1E\cdot t+S^2E\cdot t^2+S^3E\cdot t^3\ldots)=\Psi_1(E)+\Psi_2(E)\cdot t+\Psi_3(E)\cdot t^2+\ldots.$$
This equation is a definition of Adams operations.

If $\zeta_{E}(g)$ is the character of a representation $E$, then it is easy to see, that the character of a representation $\Psi_k(E)$ evaluated at an element $g$ is equal to
$$\zeta_{\Psi_k(E)}(g)=\zeta_{E}(g^k).$$
The above equations for Adams operations are transparent from this formula.

We have also used in the Getzler's formula Adams operations on the ring of polynomials of two variables $u$ and $v$. In fact, in the ring of polynomials of any number of variables one can consider
$$\Psi_k(P(x_1,\ldots, x_n))=P(x_1^k, x_2^k,\ldots, x_n^k).$$
Again all above equations are clear.

\section{Calculations}

Consider the forgetful map $\mathcal{M}_{2,n}\rightarrow \mathcal{M}_2$, sending a curve with marked points to the same curve without marks. Naively, for a curve $C$ the fiber of this map over
the point in $\mathcal{M}_2$, corresponding to a point $C$, is just
$F(C, n)$. Using the Getzler's formula, one could easily calculate the Euler characteristic of a fiber (it is the same for all $C$), and
therefore obtain an answer for $\mathcal{M}_{2,n}$. This naive approach yields wrong answer $(1+p_1t)^{-2}$, which does not coincide with known answers even for $\mathcal{M}_{2,1}$.

In fact, the situation is more complicated. The fiber of the forgetful map over a point $C$ is $F(C, n)/Aut(C)$, where $Aut(C)$
is a group of automorphisms of a curve $C$. For example, for $\mathcal{M}_{2,1}$ the fiber of the forgetful map is not $C$ with the Euler characteristic -2, but the quotient $C/Aut(C)$.
For a generic  curve the automorphism group contains only the hyperelliptic involution, and $C/Aut(C)=\mathbb{CP}^1$. For other
curves $C/Aut(C)$ is a quotient of the projective line by the finite group action, and is isomorphic to $\mathbb{CP}^1$ as well.
Therefore $\chi(\mathcal{M}_{2,1})=2.$

Suppose that a finite group $G$ acts on a variety $X$, and each  nonidentical element of $G$ has a finite number of the fixed points. Let $v(g)$ be an order of an element $g$ and $O_k(g)$ be a number of orbits of length $k$ of the action of $g$ on $X$ (of course, $1\le k\le v(g)-1$ and $k|v(g)$).

\begin{theorem}
$$\sum_{n=0}^{\infty}t^{n}\chi^{S_n}(F(X,n)/G)={1\over |G|}\sum_{g\in G}{\prod_{k|v(g)}(1+p_kt^{k})^{O_k(g)}\cdot (1+p_{v(g)}t^{v(g)})^{\chi(X)\over v(g)}\over (1+p_{v(g)}t^{v(g)})^{{1\over v(g)}\sum_{l|v(g)}l\cdot O_l(g)}}.$$
\end{theorem}

\begin{proof}
Let $R\in K_0(Rep(G))$ be the alternate sum of cohomologies of $X$ as representations of the group $G$ (belonging to the Grothendieck ring of representations of $G$). By the Getzler's formula the analogous alternate sum of $G$-representations for $S_n$-equivariant cohomologies of $F(X,n)$ equals to
$$\prod_{k=1}^{\infty}(1+p_kt^{k})^{{1\over k}\sum_{d|k}\mu(k/d)\Psi_d(R)},$$
and its character at the element $g$ equals to
$$\xi(g)=\prod_{k=1}^{\infty}(1+p_kt^{k})^{{1\over k}\sum_{d|k}\mu(k/d)\chi_R(g^d)},$$
if $\chi_R$ is a character of the representation $R$.
Note that the Lefshetz fixed points theorem states that
$$\chi_R(g)=\sum_{i=0}^{\infty}Tr(g|_{H^{i}(X)})=L(g),$$
where $L(g)$ is the Lefshetz number of the map $g$, that is, in our case $\chi_R(e)=\chi(X)$, and for other $g$ $\chi_R(g)$ equals to the number of fixed points of the map  $g$.

Moreover, $\chi^{S_n}(F(X,n)/G)$ is equal to the dimension
of the $G$-invariant part in the cohomologies of $F(x,n)$, that is
$${1\over |G|}\sum_{g\in G}\xi(g)$$.

Let $v$ be an order of $g$. If $d$ is divisible by $v$, then  $\chi_R(g^d)=\chi(X),$ otherwise the number of fixed points of  $g^d$ is equal to the number of points in orbits, whose length
is a divisor of $d$, that is $\sum_{l|d}l\cdot O_l(g).$
Therefore, $${1\over k}\sum_{d|k}\mu(k/d)\chi(g^d)={1\over k}\sum_{d|k,d\vdots v}\mu(k/d)\chi(X)+
{1\over k}\sum_{d|k,\not v|d}\mu(k/d)\sum_{l|d}l\cdot O_l(g).$$
Note that $${1\over k}\sum_{d|k}\mu(k/d)\sum_{l|d}l\cdot O_l(g)={1\over k}\sum_{l|k}l\cdot O_l(g)\sum_{d|k,d\vdots l}\mu(d)=
{1\over k}k\cdot O_k(g)=O_k(g),$$
$${1\over k}\sum_{d|k,v|d}\mu(k/d)\sum_{l|d}l\cdot O_l(g)={1\over k}\sum_{d|k,v|d}\mu(k/d)\sum_{l|v}l\cdot O_l(g)=\delta_{v,k}{1\over k}\sum_{l|v}l\cdot O_l(g),$$
$${1\over k}\sum_{d|k,d\vdots v}\mu(k/d)\chi(X)=\delta_{k,v}{1\over k}\chi(X).$$
Combining these answers, we obtain
$$\xi(g)={\prod_{k|v}(1+p_kt^{k})^{O_k(g)}\cdot (1+p_vt^v)^{{1\over v}\chi(X)}\over (1+p_vt^v)^{{1\over v}\sum_{l|v}l\cdot O_l(g)}},$$
completing the proof of the theorem.
\end{proof}

\begin{corollary}
$$\sum_{n=0}^{\infty}{t^{n}\over n!}\chi(F(X,n)/G)={1\over |G|}((1+t)^{\chi(X)}+\sum_{g\neq e}(1+t)^{L(g)}),$$
where $L(g)=O_1(g)$ is a number of fixed points of the map $g$.
\end{corollary}

\begin{remark}
It is useful to remark that the power of the binomial
 $(1+p_{v(g)}t^{v(g)})$ is integer. Indeed,
it equals ${1\over v(g)}(\chi(X)-\sum l\cdot O_l(g))$,
that is the Euler characteristic of the quotient of the space of points with  the length of orbit  $v(g)$ by the (free) action of the cyclic subgroup generated by $g$. Moreover, if we denote by
$X_k(g)$ the set of points of $X$ with $g$-orbit of length $k$,
one can rewrite Theorem 1 in the form
$$\sum_{n=0}^{\infty}t^{n}\chi^{S_n}(F(X,n)/G)={1\over |G|}\sum_{g\in G}\prod_{k|v(g)}(1+p_kt^{k})^{\chi(X_k(g)/<g>)}.$$

This answer is true without any finiteness of the fixed point set, but the initial one is simpler in use.
\end{remark}

\begin{lemma}
The Euler characteristic of a quotient of $B(\mathbb{C}^{*}, n)$
by the natural action of $\mathbb{C}^{*}$ is equal to 0 for even $n$ and  is equal to 1 for odd $n$.
\end{lemma}

\begin{proof}

We can divide all points by the first one, so one of the points becomes equal to 1. Other points belong to the set $B(\mathbb{C}\setminus\{0,1\},n-1)$ with Euler characteristic $(-1)^{n-1}$ due to the corollary from the Getzler's formula. We do not distinguish points,
so can choose first point arbitrarily from $n$ variants.

Let  $U_d$ be  a stratum in
$B(\mathbb{C}^{*},n)/\mathbb{C}^{*}$ where $d$ is a maximal symmetry order of a configuration. It is clear that

\begin{equation}
\label{eq1}
\chi(B(\mathbb{C}^{*},n)/\mathbb{C}^{*})=\sum_{d|n}\chi(U_{d}),
\end{equation}

and for all $d$

\begin{equation}
\label{eqat2}
\chi(B(\mathbb{C}\setminus\{0,1\},{n\over d}-1))=(-1)^{n/d-1}=\sum_{d_1\vdots d,\,\, d_1|n}{n\over d_1}\chi(U_{d_1}) .
\end{equation}

If we set $n/d=m$, $n/d_1=k$, we  can rewrite (\ref{eqat2}) as
$$\sum_{k|m}k\chi(U_{n/k})=(-1)^{m-1}.$$
We obtain a system of linear equations on $\chi(U_d)$ enumerated by divisors of $n$, which has a unique solution.

If $n$ is odd, then each its divisor $m$ is odd, so it is easy to see, that $U_n=1$ and $U_k=0$ for $k<n$.

If $n$ is even, it is easy to see, that $U_n=1$, $U_{n\over 2}=-1$,
and $U_k=0$ for $k<{n\over 2}$.

Now (\ref{eq1}) implies the lemma.
\end{proof}

Let us return to the calculation of the equivariant Euler characteristic of $\mathcal{M}_{2,n}$. Since every curve of genus 2 is a hyperelliptic covering over the Riemann sphere ramified in 6 points, we will describe all tuples of 6 points on the $\mathbb{CP}^1$
with nontrivial symmetry groups and compute the Euler characteristics of the moduli spaces of corresponding curves.

1. Regular pentagon and a point in its center. Each 6-tuple with a symmetry of fifth order is equivalent to this configuration by the
projective transformation. All these configurations are projectively equivalent, the moduli space consists of one point.
To compute the generating function for the  $S_n$-equivariant
Euler characteristic of fibers we will use Theorem 1.
Let us compute for every  $g$ the corresponding summand.

a) Identical map. $(1+p_1t)^{\chi(X)}=(1+p_1t)^{-2}.$

b) Hyperelliptic involution. $v(g)=2, O_1(g)=6$.
Therefore the corresponding summand is
$${(1+p_1t)^6(1+p_2t^2)^{-2/2}\over (1+p_2t^2)^{6/2}}=(1+p_1t)^6(1+p_2t^2)^{-4}.$$

Above two summands present in each stratum.

c) Rotation  by $2\pi/5$. $v(g)=5, O_1(g)=3$. Summand equals to
$${(1+p_1t)^3(1+p_5t^5)^{-2/5}\over (1+p_5t^5)^{3/5}}=(1+p_1t)^3(1+p_5t^5)^{-1}.$$

d) Rotation composed with involution. $v(g)=10, O_1(g)=1, O_2(g)=1, O_5(g)=1.$ Summand equals to
$${(1+p_1t)(1+p_2t^2)(1+p_5t^5)(1+p_{10}t^{10})^{-{2\over 10}}\over (1+p_{10}t^{10})^{{1\over 10}(1+2+5)}}=
(1+p_1t)(1+p_2t^2)(1+p_5t^5)(1+p_{10}t^{10})^{-1}.$$

Thus,
$$\sum_{n=0}^{\infty}t^n\chi^{S_n}(F(C,n)/G))={1\over 10}((1+p_1t)^{-2}+(1+p_1t)^6(1+p_2t^2)^{-4}+$$ $$+4(1+p_1t)^3(1+p_5t^5)^{-1}+4(1+p_1t)(1+p_2t^2)(1+p_5t^5)(1+p_{10}t^{10})^{-1}).$$

The Euler characteristic of this stratum is equal to 1.

2. Regular hexagon. Analogously we obtain an answer:
$$\sum_{n=0}^{\infty}t^n\chi^{S_n}(F(C,n)/G))={1\over 24}((1+p_1t)^{-2}+(1+p_1t)^6(1+p_2t^2)^{-4}+$$ $$+4(1+p_1t)^2(1+p_2t^2)(1+p_6t^6)^{-1}+2(1+p_1t)^4(1+p_3t^3)^{-2}+$$
$$+2(1+p_2t^2)^2(1+p_3t^3)^2(1+p_6t^6)^{-2}
+8(1+p_1t)^2(1+p_2t^2)^{-2}+6(1+p_1t)^2)(1+p_2t^2)^2(1+p_4t^4)^{-2}.$$

This stratum is also a one point, and its Euler characteristic equals to 1.

3. Octahedron. Every 6-tuple of points with a symmetry of 4th order is projectively equivalent to the octahedron, therefore the moduli space is one point. The symmetry group of an octahedron contains the identical map, 6 rotations of order 4 with an axis connecting two vertices, 3 rotations on  $180^{\circ}$ with the same axis,
8 rotations by $120^{\circ}$ with axis,
passing through the centers of opposite faces, and
 6 rotations by $180^{\circ}$ with axis connecting the mediums of opposite edges. Moreover, every symmetry has a pair -- its composition with the hyperelliptic involution.
Combining all summands,  one analogously gets the following expression:
$${1\over 48}((1+p_1t)^{-2}+(1+p_1t)^6(1+p_2t^2)^{-4}+12(1+p_1t)^2(1+p_4t^4)(1+p_8t^8)^{-1}+$$
$$+6(1+p_1t)^2(1+p_2t^2)^2(1+p_4t^4)^{-2}+8(1+p_1t)^4(1+p_3t^3)^{-2}+8(1+p_2t^2)^2(1+p_3t^3)^2(1+p_6t^6)^{-2}+$$
$$+12(1+p_1t)^2(1+p_2t^2)^{-2}).$$

The Euler characteristic of this stratum is equal to 1.

4. A pair of regular triangles with a common center.
The symmetry group of the corresponding curve consists of
rotations, transformations of second order exchanging triangles,
and their compositions with the hyperelliptic involution.
Taking into account all these transformations, we obtain the answer:
$${1\over 12}((1+p_1t)^{-2}+(1+p_1t)^6(1+p_2t^2)^{-4}+2(1+p_1t)^4(1+p_3t^3)^{-2}+$$
$$+2(1+p_2t^2)^2(1+p_3t^3)^2(1+p_6t^6)^{-2}+6(1+p_1t)^2(1+p_2t^2)^{-2}).$$

The map $z\mapsto z^3+{1\over z^3}$ identifies the moduli space of curves of such type with the space $B(\mathbb{C}^{*},2)/\mathbb{C}^{*}$ of Euler characteristic 0. We should exclude  a regular hexagon and an octahedron from it, so the Euler characteristic of this stratum equals to -2.

5. General central-symmetric configuration of points.  The symmetry group contains a central symmetry, hyperelliptic involution and their composition. The answer for the equivariant Euler characteristic of fiber
equals to
$${1\over 4}((1+p_1t)^{-2}+(1+p_1t)^6(1+p_2t^2)^{-4}+2(1+p_1t)^2(1+p_2t^2)^{-2}).$$

The map $z\mapsto z^2+{1\over z^2}$ identifies the moduli space of curves of such type with the space $B(\mathbb{C}^{*},3)/\mathbb{C}^{*}$ of Euler characteristic 1. We should exclude strata of configurations of types 4 and 6 and their intersection (octahedron and regular hexagon) from it, so the Euler characteristic of this stratum equals to $1-(-2)-(-2)-2=3$.

6. The central-symmetric configuration of 4 points, 0 and the infinity. The symmetry group contains  transformations of previous
point and their compositions with the transformation exchanging 0 and $\infty$, and pairs of central-symmetric points. The answer in this case equals to
 $${1\over 8}((1+p_1t)^{-2}+(1+p_1t)^6(1+p_2t^2)^{-4}+2(1+p_1t)^2(1+p_2t^2)^{2}(1+p_4t^4)^{-2}+4(1+p_1t)^2(1+p_2t^2)^{-2}).$$

The map $z\mapsto z^2+{1\over z^2}$ identifies the moduli space of curves of such type with the space $B(\mathbb{C}^{*},2)/\mathbb{C}^{*}$ of Euler characteristic 0.  We should exclude from it a regular hexagon and an octahedron, so the Euler characteristic of this stratum equals -2.

7. Asymmetric configuration. Unique nontrivial symmetry is a hyperelliptic involution.
The answer reads as
$${1\over 2}((1+p_1t)^{-2}+(1+p_1t)^6(1+p_2t^2)^{-4}).$$

The Euler characteristic of this stratum equals to -1. It can be obtained by subtraction from $\chi(\mathcal{M}_{2})=1$ (by Tommasi's theorem) of Euler characteristics of all previous strata.

\section{Results}

Let us check that the weighted sum of Euler characteristics of strata is equal to the orbifold Euler characteristic of the coarse moduli space $\mathcal{M}_2$.


The sum of Euler characteristics of strata, divided by orders of corresponding symmetry group, is equal
$${1\over 10}+{1\over 24}+{1\over 48}-{2\over 12}+{3\over 4}-{2\over 8}-{1\over 2}={24+10+5-40+180-60-120\over 240}={-1\over 240},$$
in agreement with the formula of Harer and Zagier
(\cite{harzag}):
$$\chi_{orb}(\mathcal{M}_{g,n})=(-1)^{n}{(2g-3+n)!(2g-1)\over (2g)!}B_{2g},$$
where $B_{2g}$ are Bernoulli numbers. For $g=2$ $B_{2g}={-1\over 30},$
so  $$\chi_{orb}(\mathcal{M}_{2,0})={1\cdot 3\over 4!}\cdot {-1\over 30}={-1\over 240}.$$

Since the agreement is achieved, let us collect all answers from the different strata. We get

\begin{theorem} The generating function for the $S_n$-equivariant Euler characteristics of $\mathcal{M}_{2,n}$ is equal to
$$\sum_{n=0}^{\infty}t^n\chi^{S_n}(\mathcal{M}_{2,n}))={-1\over 240}(1+p_1t)^{-2}-{1\over 240}(1+p_1t)^6(1+p_2t^2)^{-4}+$$
$$+{2\over 5}(1+p_1t)^3(1+p_5t^5)^{-1}+{2\over 5}(1+p_1t)(1+p_2t^2)(1+p_5t^5)(1+p_{10}t^{10})^{-1}+$$
$$+{1\over 6}(1+p_1t)^2(1+p_2t^2)(1+p_6t^6)^{-1}-{1\over 12}(1+p_1t)^4(1+p_3t^3)^{-2}-$$
$$-{1\over 12}(1+p_2t^2)^2(1+p_3t^3)^2(1+p_6t^6)^{-2}+{1\over 12}(1+p_1t)^2(1+p_2t^2)^{-2}+$$
$$+{1\over 4}(1+p_1t)^2(1+p_4t^4)(1+p_8t^8)^{-1}-{1\over 8}(1+p_1t)^2(1+p_2t^2)^2(1+p_4t^4)^{-2}.$$
\end{theorem}


This answer should be compared to the known results for small number of points, that is, for the coefficients at small powers of $t$.

Using the Maple program, one obtains

$$\sum_{n=0}^{\infty}t^n\chi^{S_n}(\mathcal{M}_{2,n}))=
1+2p_1\cdot t+p_1^2\cdot t^2+({1\over 2}p_4+{2\over 3}p_1p_3-{1\over 6}p_1^4)\cdot t^4+\ldots=$$
$$1+2s_1\cdot t^4+(s_1+s_2)\cdot t^2+(s_4-s_{3,1}-s_{2,2})\cdot t^4+\ldots,$$
which agrees with the  answers obtained by Getzler (\cite{TRR}) and Tommasi and Bergstr$\ddot{o}$m(\cite{bertom}).

It can be useful, also, to consider the usual (non-equivariant) Euler characteristic. To obtain it, one should put $p_1=1, p_i=0$ for $i>1$, and Theorem 2 implies that

\begin{equation}
\label{euler}
\sum_{n=0}^{\infty}{t^n\over n!}\chi(\mathcal{M}_{2,n})=-{1\over 240}(1+t)^{-2}-{1\over 12}+{2\over 5}(1+t)+{3\over 8}(1+t)^2+{2\over 5}(1+t)^3-{1\over 12}(1+t)^4-{1\over 240}(1+t)^{6}.
\end{equation}

This  gives us a list of Euler characteristics,  which can be also compared with the list of orbifold Euler characteristics known from the Harer-Zagier formula:
\vskip 1 cm
		\begin{tabular}{|l||l|l|l|l|l|l|l|l|}
			\hline
$n$ & $0$ & $1$ & $2$ & $3$ & $4$ & $5$ & $6$ & $n\ge 7$\\ \hline
$\chi(\mathcal{M}_{2,n})$ & $1$ & $2$ & $2$ & $0$ & $-4$ & $0$ & $-24$ & ${(-1)^{n+1}\cdot (n+1)!\over 240}$ \\
\hline
$\chi_{orb}(\mathcal{M}_{2,n})$ & $-{1\over 240}$ & ${1\over 120}$ & $-{1\over 40}$ & ${1\over 10}$ & $-{1\over 2}$ & $3$ & $-21$ & ${(-1)^{n+1}\cdot (n+1)!\over 240}$

\\ \hline
		\end{tabular}

\vskip 1 cm

It is easy to explain the remarkable symmetry of the equation (\ref{euler}) (the coefficient at $(1+t)^L$ is the same as the one at $(1+t)^{4-L}$): if an automorphism of a curve has $L$ fixed points,
then its composition with the hyperelliptic involution has $4-L$ fixed points. This fact is also true for a hyperelliptic curve of any genus. For the identical map Lefshetz number equals to $2-2g$, and the hyperelliptic involution has $2+2g$ fixed ramification points. For any other automorphism its projection onto $\mathbb{CP}^1$ has two fixed points, and fixed points of the initial map should belong to the preimage of these points. Given map interchanges the preimages of one of these points, if and only if its composition with the involution does not interchange them. Therefore the sum of fixed points of two maps is equal to 4.

\begin{lemma}
The generating sum for the Euler characteristics of moduli spaces of hyperelliptic curves of the given genus $g$ with marked points
has a form

$$\sum_{n=0}^{\infty}{t^n\over n!}\chi(\mathcal{H}_{g,n})=-{1\over 4g(2g+1)(2g+2)}(1+t)^{2-2g}+a+b(1+t)+c(1+t)^2+$$
\begin{equation}
\label{euler2}
b(1+t)^3+a(1+t)^4-{1\over 4g(2g+1)(2g+2)}(1+t)^{2+2g},
\end{equation}
where $a, b$ and $c$ are some unknown rational numbers (depending of $g$) satisfying the equation
\begin{equation}
\label{euler3}
2a+2b+c=1+{1\over 2g(2g+1)(2g+2)}.
\end{equation}
These Euler characteristics can be found from the following table:

\vskip 1 cm

\begin{tabular}{|l|l|}
\hline
$n$ & $\chi(\mathcal{H}_{g,n})$\\
\hline \hline
$0$ & $1$\\
\hline
$1$ & $2$\\
\hline
$2$ & $x$\\
\hline
$3$ & $3x-6$\\
\hline
$4$ & $y$\\
\hline
$5$ & $0$\\
\hline
$4<n\le 2+2g$ & $-{1\over 4g(2g+1)(2g+2)}\cdot[{(2g+2)!\over (2g+2-n)!}+(-1)^{n}\cdot{(2g-3+n)!\over (2g-3)!}]$ \\
\hline
$n>2+2g$ & ${(-1)^{n+1}\cdot (2g-3+n)!\over (2g-3)!\cdot 4g(2g+1)(2g+2)}$ \\
\hline
\end{tabular}

\vskip 1 cm

where $x$ and $y$ are some unknown integers (depending of $g$).
\end{lemma}

\begin{remark}
Since a hyperelliptic curve with more than $2g+2$ marked points has no automorphisms, in the "stable limit" $n>2g+2$ the Euler charcteristic of $\mathcal{H}_{g,n}$ coincides with the orbifold one, and therefore the answer in the last line of the table is known. The Euler characteristics from $n=0$ and $n=1$ follows directly from Tommasi's theorem. Slightly more complicated answer for $4<n\le 2+2g$, for example, the vanishing of $\chi(\mathcal{H}_{g,5})$ for each genus, seems to be a new result.
Hence we do not know the Euler characteristics only for $n=2,3,$ and 4 for higher genus, though for $g=2$ they are presented  above. \end{remark}

\begin{proof}
The above discussion and Corollary 1 imply the equation (\ref{euler2}), but with unknown (equal) coefficients at
$(1+t)^{2-2g}$ and $(1+t)^{2+2g}$. These coefficients are equal to the orbifold   Euler characteristic of the space $\mathcal{H}_{g,0}$, which equals to $-{1\over 4g(2g+1)(2g+2)}$.

Indeed, one can set three of $(2g+2)$ ramification points to be equal $0, 1,\infty$ -- this can be done in $(2g+2)(2g+1)\cdot 2g$ ways. Other points belong to the space $B(\mathbb{CP}^1\setminus\{0,1,\infty\},2g-1)$ with Euler characteristic equal to $-1$ by Getzler's formula. We should also take into account the hyperelliptic involution, which gives the factor ${1\over 2}$.
Combining these factors, we get
$${1\over (2g+2)(2g+1)\cdot 2g}\cdot (-1)\cdot {1\over 2}=-{1\over 4g(2g+1)(2g+2)}.$$

Tommasi's theorem states that $\chi(\mathcal{H}_{g,0})=1$, what implies linear equation (\ref{euler3}) on $a$, $b$, $c$.
We get

$\begin{cases}
\chi(\mathcal{H}_{g,0})=2a+2b+c+2d=1,\\
{1\over 2!}\chi(\mathcal{H}_{g,2})=6a+3b+c+2d(1+2g^2)={x\over 2!},\\
{1\over 3!}\chi(\mathcal{H}_{g,3})=4a+b+4dg^2,\\
\end{cases}$
(here $d=-{1\over 4g(2g+1)(2g+2)}$), so $${1\over 3!}\chi(\mathcal{H}_{g,3})={x\over 2!}-1.$$

\end{proof}

\section{Remarks}

Approach of this paper is similar to the Getzler's one (\cite{getzler}, \cite{getzler2} and \cite{TRR}), but it is slightly extended to the case of nontrivial automorphism groups. Since the enumeration of the symmetric configurations of points on the projective line is clear (because the classification of discrete subgroups of $PSL(2,\mathbb{C})$ is well known), one can try to obtain
the $S_n$-equivariant Euler characteristics of the moduli spaces $\mathcal{H}_{g,n}$ of hyperelliptic curves for all $g$ and $n$.
The problem of the calculation of the equivariant Euler characteristics of fibers of the forgetful map is purely combinatorial, and the geometry is involved only in the calculation of the Euler characteristics of strata with given symmetry groups. We hope that it can be calculated using methods similar to Lemma 1.

Since the classification of discrete subgroups of $PSL(2,\mathbb{C})$ give us a very clear list of possible symmetry groups, one could expect a more understandable formula for the generating function for equivariant Euler characteristics of $\mathcal{H}_{g,n}$ for all $g$ and $n$. One could expect that this will be related to the recursion relations of Bergstr$\ddot{o}$m (\cite{bergstrom}).

It is important to note, that for given $g$ one should calculate the finite  data to get the answers for all natural $n$
-- the Euler characteristics of finite number of strata in the moduli space of curves with no marked points and lengths of orbits of nontrivial automorphisms. It is also true for $\mathcal{M}_{g,n}$, but for the higher genus the symmetry groups are much more complicated.

Moscow State University,\newline
Department of Mathematics and Mechanics.\newline
E.mail: gorsky@mccme.ru.

\end{document}